\documentclass{amsart}
\usepackage{amssymb, graphicx}
\usepackage[mathscr]{eucal}
\usepackage{slashbox}
\usepackage{rotating}
\usepackage{diagbox}

\theoremstyle{plain}
\newtheorem{prop}{Proposition}[section]
\newtheorem{coro}[prop]{Corollary}

\newtheorem{lemm}[prop]{Lemma}
\newtheorem{ques}[prop]{Question}
\newtheorem{thm}[prop]{Theorem}
\newtheorem{ex}[prop]{Example}
\theoremstyle{definition}
\newtheorem{defn}[prop]{Definition}

\newtheorem{rem}[prop]{Remark}

\DeclareMathOperator{\maxdeg}{maxdeg}
\DeclareMathOperator{\mindeg}{mindeg}

\def\mcg#1;#2{\Gamma_{#1,#2}}
\def\fg#1;#2{\Pi_{#1,#2}}
\def\tb#1;#2{\mathscr{K}_{\frac{#1}{#2}}}

\begin{document}

\title[Jones Polynomial versus Determinant of Quasi-Alternating Links]
{Jones Polynomial versus Determinant of Quasi-Alternating Links}


\author{Khaled Qazaqzeh}
\address{Permanent Address: Department of Mathematics, Faculty of Science,
Yarmouk University, Irbed, Jordan, 21163}
\email{qazaqzeh@yu.edu.jo}
\urladdr{https://faculty.yu.edu.jo/qazaqzeh}

\address{Current Address: Department of Mathematics, Faculty of Science,  Kuwait
University, P. O. Box 5969 Safat-13060, Kuwait, State of Kuwait}
\email{khaled.qazaqzeh@ku.edu.kw}
\date{05/08/2022}
 
\begin{abstract}
We prove that there are only finitely many values of the Jones polynomial of quasi-alternating links of a given determinant. 
Consequently, we prove that there are only finitely many quasi-alternating links of a given Jones polynomial iff there are only finitely many quasi-alternating links of a given determinant. 
\end{abstract}

\maketitle

\section{introduction}
The study of the class of alternating links has played a crucial role in the advancement and evolution of knot theory in the last century. In particular, some long lasting conjectures about alternating links have been solved using the study of the Jones polynomial. Therefore, it is very natural to study the Jones polynomial of any class of links in trying to get a better understanding of the given class.

It is easy to see that there are only finitely many alternating links of a given determinant. This is true as a result of the fact that there are only finitely many links of a crossing number less than or equal to some positive integer and the result due to Bankwitz in \cite{B} that the determinant in the class of alternating links represents an upper bound of the crossing number. This implies that there are only finitely many values of the Jones polynomial of alternating links of a given determinant. 
Moreover, we can conclude that there are only finitely many alternating links of a given Jones polynomial as a consequence of the previous result and the well-known fact that the determinant is the absolute value of the evaluation of the Jones polynomial at $t=-1$.

In trying to naturally generalize the class of alternating links, the authors in \cite{OS} defined the class of quasi-alternating links. 
Contrary to alternating links that admits a simple diagrammatic definition, this class has been defined recursively as follows:

\begin{defn}\label{def}
The set $\mathcal{Q}$ of quasi-alternating links is the smallest set
satisfying the following properties:
\begin{itemize}
	\item The unknot belongs to $\mathcal{Q}$.
  \item If $L$ is a link with a diagram $D$ containing a crossing $c$ such that
\begin{enumerate}
\item both smoothings of the diagram $D$ at the crossing $c$, $L_{0}$ and $L_{1}$ as
in Figure \ref{figure} belong to $\mathcal{Q}$;
\item $\det(L_{0}), \det(L_{1}) \geq 1$;
\item $\det(L) = \det(L_{0}) + \det(L_{1})$; then $L$ is in $\mathcal{Q}$ and in this
case we say $L$ is quasi-alternating at the crossing $c$ with quasi-alternating
diagram $D$.
\end{enumerate}
\end{itemize}
\end{defn}

\begin{figure} [h]
\begin{center}
\includegraphics[scale=0.4]{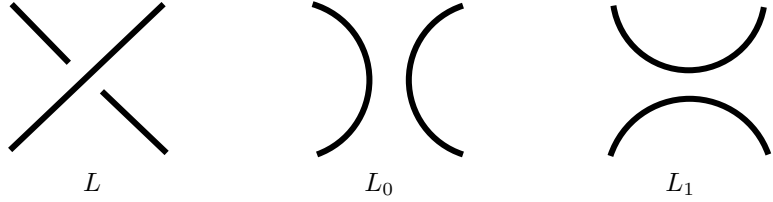} \\
{$L$}\hspace{3.5cm}{$L_0$}\hspace{3.6cm}$L_{1}$
\end{center}
\caption{The link diagram $L$ at the crossing $c$ and its smoothings $L_{0}$ and
$L_{1}$ respectively.}\label{figure}
\end{figure}

It is easy to apply this definition to determine if a given link diagram is quasi-alternating and thus the link represented by this diagram is quasi-alternating. But it can not be used to determine if a given link is not quasi-alternating because of the unlimited number of link diagrams the represent the same link. Rather than that, several obstructions for a link to be quasi-alternating have been introduced through the past two decades. Many of these  have been established  for the class of alternating links first, then generalized later to quasi-alternating links. Some of these  main  obstructions are listed here.

\begin{enumerate}
\item the branched double-cover of any quasi-alternating link
is an $L$-space \cite[Proposition.\,3.3]{OS};
\item the branched double-cover of any
quasi-alternating link bounds a negative definite $4$-manifold $W$ with
$H_{1}(W) = 0$ \cite[Proof of Lemma.\,3.6]{OS};
\item the $\mathbb{Z}/2\mathbb{Z}$ knot Floer homology group of any
quasi-alternating link is $\sigma$-thin \cite[Theorem.\,2]{MO};
\item the reduced ordinary Khovanov homology group of any
quasi-alternating link is $\sigma$-thin \cite[Theorem.\,1]{MO};
 \item the reduced odd Khovanov homology group of any
quasi-alternating link is $\sigma$-thin \cite[Remark after Proposition.\,5.2]{ORS};
 \item the determinant of any quasi-alternating link is bigger  than  the degree of
its $Q$-polynomial \cite[Theorem\,2.2]{QC2}. This inequality was later sharpened to the determinant
minus one is bigger  than or equal to the degree of the $Q$-polynomial with
equality holds only for $(2,n)$-torus links \cite[Theorem\,1.1]{T1}.
\item The length of any gap in the differential grading of the Khovanov homology of any quasi-alternating link is at most  one \cite[Theorem\,2.5]{QC1}.
\end{enumerate}

Two natural questions can be raised at this time in trying to generalize the previous discussion to the class of quasi-alternating links as follows:
\begin{ques}\label{question}
\begin{enumerate}
\item Is it true that there are only finitely many values of the Jones polynomial of quasi-alternating links of a given determinant?
\item Is it true that there are only finitely many quasi-alternating links of a given Jones polynomial?
\end{enumerate}
\end{ques}

These two question clearly have a positive answer if we assume that there are only finitely many quasi-alternating links of a given determinant similar to the case of alternating links. The positive answer for the second question under this assumption follows directly as a consequence of the fact that the determinant of any link is equal to the absolute value of the evaluation of the Jones polynomial at $t=-1$. 

The finiteness of quasi-alternating links of a given determinant was pointed out for the first time in Conjecture 3.3 in \cite{G}. This conjecture has been verified for some values of the determinant in \cite{G,LS, Te} and  was later generalized to a conjecture that states that the determinant is an upper bound of the crossing number in the class of quasi-alternating links in \cite[Conjecture\,1.1]{QQJ}. This later conjecture has been verified for several subclasses of quasi-alternating links. These two conjectures are still open as far as we know.

In this paper, we prove that the first question has a positive answer without assuming the finiteness of quasi-alternating links of a given determinant. Also as a direct consequence, we prove that the second question has a positive answer iff the finiteness of quasi-alternating links of a given determinant assumption holds. In particular, our main goal in this paper is to prove the following theorem and give some applications. 

\begin{thm}\label{main}
There are only finitely many values of the Jones polynomial of quasi-alternating links of a given determinant.
\end{thm}

 This paper is organized as follows. In Section 2, we recall the  definition of the Jones polynomial in terms of the Kauffman bracket and review some of its properties that will be used in the next section. In Section 3, we give the proof of our main result and give some applications and consequences. 

\section{The Jones polynomial}

All the results of this paper are restricted to the case where the link $L$ as illustrated in Figure \ref{figure}. Similar results can be obtained for the other case by interchanging $A$ and $A^{-1}$ in the third formula in Definition \ref{kauffman}.

The Jones polynomial $V_L(t)$ is an invariant of oriented links that was defined for the first time in \cite{J}. It is a Laurent polynomial with integral coefficients that can be defined in several ways. In this
subsection, we shall briefly recall the definition of this polynomial in terms of the Kauffman
bracket as it has been defined in \cite{Ka}.

\begin{defn}\label{kauffman}
The Kauffman bracket polynomial is a function from the set of unoriented link
diagrams in the oriented plane to the ring of Laurent polynomials with integer
coefficients in an indeterminate $A$. It maps a link $L$ to $\left\langle L
\right\rangle\in \mathbb Z[A^{-1},A]$ and it is defined by the following relations:
\begin{enumerate}
\item $\left\langle \bigcirc \right\rangle=1$,
\item $\left\langle \bigcirc \cup L\right\rangle=(-A^{-2}-A^2)\left\langle L
\right\rangle$,
\item $\left\langle L\right\rangle=A\left\langle L_0\right\rangle+A^{-1}\left\langle
L_1\right\rangle$,
\end{enumerate}
where $\bigcirc$, in relation 2 above,  denotes a trivial circle  disjoint from the rest of the link, and  $L,L_0, \text{and } L_1$  represent three unoriented links
which  are identical everywhere  except in a small region where they are as indicated in  Figure \ref{figure}.
\end{defn}
For any oriented link diagram  $L$, we let $x(L)$ denote the number of negative crossings and  $y(L)$ denote the number of positive crossings in $L$, see Figure \ref{Diagram1}.
Also, we define the writhe of $L$ to be the integer $w(L) = y(L) - x(L)$. It is easy to see that the writhes of the links $L,L_{0}$, and $L_{1}$ are related as follows: 
\begin{rem}\label{simple2}
\begin{enumerate}
\item If the crossing is positive, then we have $x(L_{0}) = x(L), y(L_{0}) = y(L) - 1, x(L_{1}) = x(L) + e$ and
$y(L_{1}) = y(L)-e-1$. Therefore, $w(L_{0}) =w(L)-1$ and $w(L_{1}) = w(L) -2e-1$,
where $e$ denotes the difference between the number of negative crossings in
$L_{1}$ and the number of negative crossings in $L$.
\item If the crossing is negative, then we have $x(L_{0}) = x(L)-1, y(L_{0}) = y(L), x(L_{1}) = x(L) - e-1$ and
$y(L_{1}) = y(L)+e$. Therefore, $w(L_{0}) =w(L)+1$ and $w(L_{1}) = w(L) +2e+1$,
where $e$ denotes the difference between the number of positive crossings in
$L_{1}$ and the number of positive crossings in $L$.
\end{enumerate}
\end{rem}
\begin{defn}\label{jones}
The Jones polynomial $V_{L}(t)$ of an oriented link $L$ is the Laurent polynomial
in $t^{\pm 1/2}$ with integer coefficients defined by
\begin{equation*}
V_{L}(t)=((-A)^{-3w(L)}\left\langle L \right\rangle)_{t^{1/2} = A^{-2}}\in
\mathbb Z[t^{-1/2},t^{1/2}],
\end{equation*}
where $\left\langle L \right\rangle $ is the Kauffman bracket of the unoriented link obtained from $L$ by ignoring  the orientation.
\end{defn}

\begin{figure}[h]
	\centering
		\reflectbox{\includegraphics[scale=0.15]{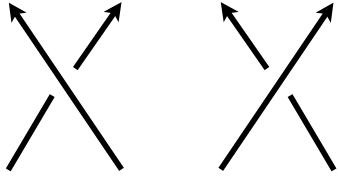}}\hspace{1.1cm}
		\includegraphics[scale=0.15]{Diagram1}
	\caption{Negative and positive crossings respectively.}
	\label{Diagram1}
\end{figure}

The following lemma is a direct consequence of Remark \ref{simple2} and it gives the relation between the Jones polynomial of the link and the Jones polynomial of the links $L_{0}$ and $L_{1}$. 
\begin{lemm}\label{jonespolynomial}
The Jones polynomial of the link $L$ at the crossing $c$ satisfies one of the following skein relations:
\begin{enumerate}
\item If $c$ is a positive crossing, then
$ V_{L}(t) = -t^{\frac{1}{2}}V_{L_{0}}(t)-t^{\frac{3e}{2}+1}V_{L_{1}}(t)$.
\item If $c$ is a negative crossing, then
$ V_{L}(t) = -t^{\frac{3e}{2}-1}V_{L_{0}}(t)-t^{\frac{-1}{2}}V_{L_{1}}(t)$.
\end{enumerate}
where $e$ has been defined in Remark \ref{simple2}.
\end{lemm}

We close this section by the following lemma that will be used in the proof of the main theorem.

\begin{lemm}\label{simple3}
The absolute value of any coefficient of any quasi-alternating link is bounded above by the determinant.
\end{lemm}
\begin{proof}
The main result of \cite{MO} yields that the Jones polynomial of any quasi-alternating link is alternating. In other words, if $V_{L}(t) = t^{m}\sum_{i=0}^{n} a_{i}t^{i}$ then $a_{i}a_{i+1} \leq 0$ for $0\leq i \leq n-1$. Now the result follows easily since the determinant of such a link is equal to the absolute value of the alternating sum of the coefficients of the Jones polynomial. 
\end{proof}

\section{Proof of the Main Theorem and Some Consequences}

We start this section by introducing a new terminology to simplify notation for this section. We can introduce an equivalence relation on the class of quasi-alternating links for any link invariant according to the following definition.

\begin{defn}
For any link invariant, we define a relation on the class of quasi-alternating links by two quasi-alternating links are related in this invariant relation if they have the same value of this link invariant. 
\end{defn}

\begin{rem}
It is easy to see that such a relation is an equivalence relation for any link invariant and its equivalence classes will be called by the given link invariant.
\end{rem}
\begin{proof}[\textbf{Proof of Theorem \ref{main}}]
We apply induction on the determinant of the given quasi-alternating link to prove the statement of the
theorem. If $\det(L)=1$, then the result holds since the only
quasi-alternating link of determinant one is the unknot according to the definition.

Now, suppose that the result holds when the determinant is less than or equal to $n$. 
In other words, there is a finite list of Jones polynomial equivalence classes of  quasi-alternating links of determinant less than or equal to $n$. In particular, the list $\{K_{1}, K_{2}, \ldots, K_{m}\}$ consists of all Jones polynomial equivalence classes of quasi-alternating links of determinant less than or equal to $n$. In particular, if $K$ is a quasi-alternating link of determinant less than or equal to $n$, then $V_{K}(t) = V_{K_{i}}(t)$ for some $1\leq i \leq m$.

The definition of $L$ being quasi-alternating of determinant $n+1$ implies that the links $L_{0}$ and $L_{1}$ have determinant less than $n+1$. By a direct application of the induction hypothesis on the links $L_{0}$ and $L_{1}$, we have only finitely many values of $V_{L_{0}}(t)$ and $V_{L_{1}}(t)$ since $L_{0}$ and $L_{1}$ represent two Jones polynomial equivalence classes in the above list.

Now the Jones polynomial of the link $L$ is given by one of the two formulas in Lemma \ref{jonespolynomial} in terms of the Jones polynomial of the links $L_{0}$ and $L_{1}$.  The choice of the formula depends on the sign of the given crossing. 

We let $m =\min \{\mindeg V_{K_{i}}(t) \ | \ 1\leq i \leq n\}$ and $M =\max \{ \maxdeg V_{K_{i}}(t) \ | \ 1\leq i \leq n\}$ for the two cases of the proof according to Lemma \ref{jonespolynomial}:
\begin{enumerate}
\item   For the first case, we have $ V_{L}(t) = -t^{\frac{1}{2}}V_{L_{0}}(t)-t^{\frac{3e}{2}+1}V_{L_{1}}(t)$ with only finitely many values of the number $e$. Otherwise this will force consecutive coefficients of $V_{L}(t)$ to vanish and this contradicts with the result of \cite[Corollary\,4.1]{QC1}. This argument implies that there will be only finitely many values of $V_{L}(t)$ since $\mindeg V_{L}(t) \geq m - \left|\frac{3e}{2}+1\right|$ and $\maxdeg V_{L}(t) \leq M + \left|\frac{3e}{2}+1\right|$ 
and the absolute value of the coefficients of $V_{L}(t)$ are bounded above by $n+1$ as a result of Lemma \ref{simple3}. 
\item  For the second case, we have $ V_{L}(t) = -t^{\frac{3e}{2}-1}V_{L_{0}}(t)-t^{\frac{-1}{2}}V_{L_{1}}(t)$ with only finitely many values of the number $e$. Otherwise this will force consecutive coefficients of $V_{L}(t)$ to vanish and this contradicts with the result of \cite[Corollary\,4.1]{QC1}. This argument implies that there will be only finitely many values of $V_{L}(t)$ since $\mindeg V_{L}(t) \geq m - \left|\frac{3e}{2}-1\right|$ and $\maxdeg V_{L}(t) \leq M + \left|\frac{3e}{2}-1\right|$ 
and the absolute value of the coefficients of $V_{L}(t)$ are bounded above by $n+1$ as a result of Lemma \ref{simple3}. 
\end{enumerate}

\end{proof}

\begin{coro}\label{khovanov}
There are only finitely many Jones polynomial equivalence classes of quasi-alternating links of determinant less than or equal to any positive number. Moreover, we conclude that there are only finitely many Khovanov homology equivalence classes of quasi-alternating links of determinant less than or equal to any positive number. 
\end{coro}

\begin{rem}
We can conclude that the breadth of the Jones polynomial of any quasi-alternating link in a determinant equivalence class is bounded above by some positive integer. This upper bound clearly depends entirely on the determinant of the equivalence class and it is conjectured to be the determinant itself in \cite[Conjecture\,3.8]{QC}.
\end{rem}

\begin{coro}
In any infinite family of links of the same determinant, either this family has only finitely many values of the Jones polynomial or it has only finitely many quasi-alternating links.
\end{coro}

\begin{ex}
Kanenobu in \cite{K} constructed an infinite family of knots of the same determinant that have infinitely many values of the Jones polynomial. Therefore, we can conclude that this family contains only finitely many quasi-alternating links which confirms the result of \cite[Corollary\,3.3]{QC2}. Moreover, this implies that the result of Theorem \ref{main} does not hold in the class of all links.
\end{ex}

One might wonder if the Jones polynomial equivalence classes are finite in the class of quasi-alternating links. It is easy to see that some of the Jones polynomial equivalence classes are finite. For example, the Jones polynomial equivalence classes of $1, -t^{\frac{-1}{2}}-t^{\frac{-5}{2}}, -t^{-4}+t^{-3}+t^{-1}, -t^{\frac{-9}{2}}-t^{\frac{-5}{2}}+t^{\frac{-3}{2}}-t^{\frac{-1}{2}},t^{-2}-t^{-1}+1-t+t^{2},-t^{-7}+t^{-6}-t^{-5}+t^{-4}+t^{-2}, -t^{\frac{-17}{2}}+t^{\frac{-15}{2}}-t^{\frac{-13}{2}}+t^{\frac{-11}{2}}-t^{\frac{-9}{2}}-t^{\frac{-5}{2}}, -t^{-10}+t^{-9}-t^{-8}+t^{-7}-t^{-6}+t^{-5}+t^{-3}$ and $-t^{-6}+t^{-5}-t^{-4}+2t^{-3}-t^{-2}+t^{-1}$ are all finite since there are only finitely many quasi-alternating links of determinant $1,2,3,4,5,$ and $7$, respectively according to the results in \cite{G,LS,Te}. 

As a consequence of the main result and based on the previous discussion, a partial answer to Question \ref{question}(2) can be given in the following corollary:

\begin{coro}\label{easy}
In the class of quasi-alternating links, every determinant equivalence class is finite iff every Jones polynomial equivalence class is finite.
\end{coro}

\begin{proof}
The first implication follows directly as a result of the fact that every Jones polynomial equivalence class consists of links of the same determinant.
Now the second implication follows since each determinant equivalence class is a finite union of Jones polynomial equivalence classes as a result of having only finitely many values of the Jones polynomial of quasi-alternating links of a given determinant. 
\end{proof}

\begin{rem}
The result in Corollary \ref{easy} still holds if Jones polynomial is replaced by Khovanov homology based on the result of Corollary \ref{khovanov}.
\end{rem}
Moreover, we can generalize Question \ref{question}(2) to any other link invariant that determines the determinant as follows: 
\begin{ques}
In the class of quasi-alternating links and any link invariant that determines the determinant, is it true that all invariant equivalence classes are finite?
\end{ques} 



The next result uses the twisting technique that was first
introduced by Champanerkar and Kofman \cite[Page\,2452]{CK}. It basically replaces the 
crossing $c$ of the link $L$ by some alternating integer tangle. The original technique uses not only integer tangle but also a rational 
tangle and it is applied only at the crossing where the given link is quasi-alternating. Now we can exhibit an infinite family of non quasi-alternating links from a given particular link using the twisting technique as follows: 

\begin{coro}\label{simple}
Let $L^{n}$ be the link obtained by replacing the crossing $c$ of the link $L$ by an alternating integer tangle of $n$ crossings such that the type of each crossing is of the same type as the type of $c$. 
If $\det(L_{0}) =0$ or $\det(L_{1}) =0$ and $V_{L_{0}}(t)$ and $V_{L_{1}}(t)$ are nonzero,  then $\{L^{n}| n \in \mathbb{N}\}$ contains only finitely many quasi-alternating links.
\end{coro}
\begin{proof}
The result follows since $\det(L^{n}) = \det(L)$ for any $n$ and $V_{L^{i}}(t) \neq V_{L^{j}}(t)$ for $i \neq j$ according to the formulas in Lemma \ref{jonespolynomial}. This infinite family of links consists of links of the same determinant with infinitely many values of the Jones polynomial. Thus, only finitely many links of this family are quasi-alternating. 
\end{proof}

As a direct application of the above corollary, we  give the following example: 
\begin{ex}
Let $K_{l}^{m} = K_{\beta}(\sigma^{2l},\sigma^{-2m})$ with $\beta = \sigma_{1}^{-1}\sigma_{2}\sigma^{-2n}$ be the knot defined in \cite{LR} and used later in \cite{W} for any positive integers $l$ and $m$. It is not too hard to prove that  $\det(K_{l}^{m}) = (\det(\beta))^{2}$ and $V_{K_{l}^{m}}(t) \neq V_{K_{l^{'}}^{m}}(t)$ for $l \neq l^{'}$. Then according to the result of Corollary \ref{simple}, the family of knots $\{K_{l}^{m}  |  \text{for fixed} \ m \}$ contains only finitely many quasi-alternating links. 
\end{ex}

\end{document}